\documentclass[10pt,a4paper,oneside]{article}
\usepackage {enumerate,amsthm,amsmath,amsfonts,psfrag}
\usepackage {authblk}
\usepackage {latexsym}
\usepackage {graphics,graphicx}
\usepackage {fancyhdr}
\usepackage [usenames,dvipsnames]{xcolor}
\usepackage {empheq}
\usepackage {subfigure}

\usepackage [colorlinks=true, a4paper=true, pdfstartview=FitV,linkcolor=blue, citecolor=blue, urlcolor=blue]{hyperref}%gives clickable contents
\usepackage {datetime}
\usepackage {natbib} %natural bibliography

\newcommand {\startenv} {\vskip 0.3ex \begin{tabular}{||l}\parbox[t]{0.90\linewidth}}
\newcommand {\stopenv} {\end{tabular}\vskip 0.3ex}

\theoremstyle{plain}
    \newtheorem{theorem}{Theorem}[section]
    \newtheorem{lemma}{Lemma}[section]

\theoremstyle{definition}   %this style exists in the amsthm package
    \newtheorem{definition}{Definition}[section]
    \newtheorem*{definition*}{Definition}
\theoremstyle{remark}       %this	style exists in the amsthm package
    \newtheorem{remark}{Remark}[section]

    \newtheorem*{requirement*}{\underline{Requirement}}
    \newtheorem{example}{Example}
\title{Entropy dissipation of moving mesh adaptation}
\author[1]{\normalsize{Maria Luk{\'a}{\v{c}}ov{\'a}-Medvi{d'}ov{\'a}}} 	\author[1]{Nikolaos Sfakianakis}
\affil[1]{Johannes Gutenberg University, Mainz, Germany}
\date{}
\begin{document}
\maketitle
\begin{abstract}
	Non-uniform grids and mesh adaptation have been a growing part of numerical simulation over the past years. It has been experimentally noted that mesh adaptation leads not only to locally 
	improved solution but also to numerical stability of the underlying method. There have been though only few results on the mathematical analysis of these schemes (see for example 
	\cite*{Sfakianakis.2008}) due to the lack of proper tools that incorporate both the time evolution and the mesh adaptation step of the overall algorithm. 

	In this paper we provide a method to perform the analysis of the mesh adaptation method, including both the mesh reconstruction and evolution of the solution. We moreover employ this method to 
	extract sufficient conditions -on the adaptation of the mesh- that stabilize a numerical scheme in the sense of the entropy dissipation. 
\end{abstract}

%%%%%%%%%%%%%%%%%%%%%%%%%%%%%%%%%%%%%%%%%%%%%%%%%%%%%%%%%%%%%%%%%%%%%%%%%%%%%%%%%%%%%%
%%%%%%%%%%%%%%%%%%%%%%%%%%%%%%%%%%%%%%%%%%%%%%%%%%%%%%%%%%%%%%%%%%%%%%%%%%%%%%%%%%%%%%
%%%%%%%%%%%%%%%%%%%%%%%%%%%%%%%%%%%%%%%%%%%%%%%%%%%%%%%%%%%%%%%%%%%%%%%%%%%%%%%%%%%%%%
%%%%%%%%%%%%%%%%%%%%%%%%%%%%%%%%%%%%%%%%%%%%%%%%%%%%%%%%%%%%%%%%%%%%%%%%%%%%%%%%%%%%%%

\section{Introduction}
	Hyperbolic conservation laws appear in various applications. For example, fundamental physical laws, the conservation of mass momentum and energy, lead to the Euler equations of gas dynamics. 
	Further examples arise in traffic flows, shallow water flows, magnetohydrodynamics and biology, see, e.g. \cite*{LeVeque.2002, Toro.2001, Godlewski.1990, Sfakianakis.2012}.

	Let us consider a scalar conservation law in one space dimension,
	\begin{equation}\label{SCL}
		u_t+f(u)_x=0,\quad x\in[a,b],\ t\in[0,T],
	\end{equation}
	with initial data $u_0\in L^\infty([a,b])$. In order to simplify the presentation we assume e.g. periodic boundary conditions.

	Adaptivity is a main theme in modern scientific computing of complex physical phenomena. It is important to investigate the behaviour of adaptive schemes for hyperbolic problems, 
	such as (\ref{SCL}), which exhibit several interesting and not trivial characteristics. In this work we 	study the behaviour of certain geometrically driven adaptive algorithms when combined 
	with the important class of entropy conservative schemes introduced by \cite*{Tadmor.1987, Tadmor.2003}.

	In every time step $t=t^n$ the mesh we consider is:
	$$M_x^n=\left\{a=x_1^n<\cdots<x_N^n=b\right\}$$
	with variable space step sizes $h_i^n=x_{i+1}^n-x_i^n$, $i=1 \ldots N-1$; the mesh is reconstructed in every time step $t^n$. Further, we consider a numerical approximation $U^n$ of the exact 
	solution $u$ over the mesh $M_x^n$ at time $t=t^n$ given as a
	$$U^n=\left\{u_1^n, \ldots, u_N^n\right\}.$$
	The construction and evolution of our non-uniform meshes and the time evolution of the approximate solutions is dictated by the Main Adaptive Scheme (MAS) which is described by the 
	following procedure:
 	\begin{itemize}
		\item in every time step, construct new mesh according to the prescribed adaptivity criterion,
		\item reconstruct the numerical solution over the new mesh,
		\item evolve the numerical solution in time using the numerical scheme.
	\end{itemize}
	MAS will be discussed in details in Section \ref{sec:MAS}; we note here that the number of spatial nodes is fixed and that the reconstruction of the mesh is realized by moving its points according 
	to the geometry of the numerical solution. 

	The use of non-uniform adaptively redefined meshes, in the context of finite differences, was first studied, among others, by \cite*{Harten.1983}, \cite*{Dorfi.1987}, and \cite*{Tang.2003}. 
	The approach that we follow, for the mesh reconstruction step of MAS (Step 1) was first introduced by \cite*{Arvanitis.2001} and by \cite*{Arvanitis.2002}.  Applications of MAS on several problems, 
	point out a strong stabilisation property emanating from the mesh reconstruction \cite*{Arvanitis.2004}, \cite*{Arvanitis.2006}, \cite*{Sfakianakis.2006}, \cite*{Sfakianakis.2008}, \cite*{Sfakianakis.2009}. 
	These stabilization properties led \cite*{Sfakianakis.2006} to combine MAS with the marginal class of entropy conservative schemes. The later were first introduced by \cite*{Tadmor.1987} 
	and further studied by \cite*{LeFloch.2000, Tadmor.2003, Lukacova_Tadmor.2008}. They are semi-discrete numerical schemes which satisfy an exact entropy equality. On one hand these 
	schemes are interesting on their own right, since they appear in the context of zero dispersion limits, complete integrable systems and computation of non-classical shocks. On the other hand they 
	are important as building blocks for the construction of entropy stable schemes \cite*{Tadmor.1987, Tadmor.2003}, \cite{Lukacova_Tadmor.2008}.

	We note that classical techniques for the analysis of numerical schemes, only include the time evolution step of the procedure. In order though to have a complete picture of the quality of the numerical 
	solution under a mesh adaptation procedure a broader analysis is needed. In this direction the work done in \cite{Sfakianakis.2009} has provided some constructive analysis tools. In the present paper 
	though we are able to combine in one relation the effect of both the time evolution and the mesh adaptation. Let us point out that our approach allows us to represent the effects of both the adaptive mesh
	reconstruction as well as finite volume scheme in one conservative update relation over a reference uniform mesh \eqref{CombNumScheme}. This approach allows to apply the entropy stability analysis 
	and derive a sufficient mesh adaptation criterion to control entropy production.

%%%%%%%%%%%%%%%%%%%%%%%%%%%%%%%%%%%%%%%%%%%%%%%%%%%%%%%%%%%%%%%%%%%%%%%%%%%%%%%%%%%%%%
%%%%%%%%%%%%%%%%%%%%%%%%%%%%%%%%%%%%%%%%%%%%%%%%%%%%%%%%%%%%%%%%%%%%%%%%%%%%%%%%%%%%%%

\section{Main  Adaptive Scheme (MAS)}\label{sec:MAS}~
	Using non-adaptive meshes -both uniform and non-uniform- the evolution of the numerical solution is dictated solely by the solution update. On the contrary, in the adaptive 
	mesh case, two more phenomena need to be taken into account; the construction of the new mesh and the solution update. These steps comprise the Main Adaptive Scheme (MAS):
	\begin{definition}[MAS]\label{MAS}
		Given mesh $M_x^n$ and approximations $U^n$,
		\begin{enumerate}
			\item (Mesh Reconstruction). Construct new mesh $M_x^{n+1}$
			\item (Solution Update). Reconstruct $U^n$ over $M_x^{n+1}$ to obtain $\hat U^n$.
			\item (Time Evolution). Evolve $\hat U^n$ in time to compute $U^{n+1}$ over $M_x^{n+1}$.
		\end{enumerate}
	\end{definition}

	It is important to note that in the case of a fixed mesh, uniform or non-uniform, there is no need for mesh reconstruction (Step 1.) and effectively no need for Step 2. In such 
	case MAS reduces to just the time evolution step (Step 3.) which is what is usually considered as a numerical scheme. The extra steps 
	of the adaptive MAS, on one hand, change significantly both the computation and the analysis of the numerical approximations, and on the other hand are responsible 
	for stabilization properties of the  MAS.

	In the mesh reconstruction step (Step 1.) the mesh nodes are relocated according to the	geometric information contained in the discrete numerical solution. The basic idea is geometric: 
	\begin{center}
		\emph{in areas where the numerical solution is more smooth/flat the density of the nodes is low, in areas where the numerical solution is less smooth/flat the density of the nodes should be higher.}
	\end{center}
	In fact, the mesh reconstruction process can be chosen in any suitable way. One possibility is to use the monitor function which reflects the curvature of the numerical solution.
	For details we refer \cite{Sfakianakis.2009, Arvanitis.2006}.

	Further, we consider the solution update procedure (Step 2. of MAS). The numerical solution $U^n$ is given on the old mesh $M_x^n$ and is recomputed as $\hat U^n$ on the new mesh 
	$M_x^{n+1}$. There are many ways for the reconstruction. In this work we use conservative piecewise constant reconstructions. 

	Finally, for the time evolution step (Step 3. of  MAS) we use any numerical scheme valid for non-uniform meshes. Denoting the mesh-solution pair by 
	$U^n=\left \{u_i^n,i=1\ldots N\right\}$, $M_x^n=\left\{C_i^n, |C_i^n|=h_i^n,i=1\ldots N \right\}$, we obtain in
	the case of finite volume scheme 
	\begin{equation}\label{NonUniScheme}
		u_i^{n+1}=\hat{u}_i^n-\frac{\Delta t}{h_i^{n+1}}\left(\hat F_{i+1/2}^n-\hat F_{i-1/2}^n\right).
	\end{equation}
	Here $\hat U^n=\left\{\hat u_i^n,i=1\ldots N\right\}$ is a reconstructed $U^n$ over $M_x^{n+1}$ and the numerical flux $F$ is decorated with $\hat ~$ since it is computed over the updated values $\hat U$. The numerical flux itself can be any numerical flux valid for non-uniform grids.
	We refer to \cite{Sfakianakis.2009} and \cite{Arvanitis.2006} for more details regarding both the implementation of numerical schemes over non-uniform meshes and their properties. 

\subsection{Reference uniform mesh}
	A schematic representation of MAS \eqref{MAS} in the form of mesh-solution pairs is the following
	\begin{equation}\label{2grid1}
		\big\{ M_x^n,U^n \big\}\xrightarrow{\text{mesh adapt.}}\big\{M_x^{n+1},\hat U^n\big\}\xrightarrow{\text{num. scheme}} \big\{M_x^{n+1},U^{n+1}\big\},
	\end{equation}
	where in the first part we have considered the Steps 1 and 2 of MAS and in the second part the Step 3.

	In parallel to MAS and \eqref{2grid1} we define a new set of mesh-solution pairs where the meshes are uniform, constant in time, of the same cardinallity as $M_x^n$ and discretizing the same 
	physical domain.

	\begin{definition}[Reference uniform mesh-solution pair]\label{RefMesh}
		Let $\{M_x,U\}$ and $\{\bar M, V\}$ be two mesh-solution pairs with $M_x=\{C_i,|C_i|=h_i,\; i=1 \ldots N\}$, $\bar M=\{\bar C_i, |\bar C_i|=\Delta x,\; i=1 \ldots  N\}$, 
		$U=\{u_i, i=i \ldots N\}$, and $V=\{v_i,i=1 \ldots N\}$. We call $\{\bar M, V\}$ the reference uniform mesh-solution pair to $\{M_x,U\}$ if 
		\begin{itemize}
			\item the meshes $M_x$ and $\bar M$ discretize the same physical domain, and
			\item the following per-cell mass conservation is satisfied for every $i=1, \ldots, N$
				\begin{equation}\label{PerCellCons1}
					\Delta x\, v_i=h_i\, u_i.
				\end{equation}
		\end{itemize}
	\end{definition}
	We prove in Lemma \ref{PerCellLemma}  that the per-cell conservation property \eqref{PerCellCons1} is a result of a geometric conservation law.

\subsubsection*{Geometric conservation law}
	Let us consider a a time dependent cell $C(\tau)=(x_1(\tau),x_2(\tau))$. We look for an appropriate conservation law,  
	\begin{equation}\label{GCL1}
		u_t(x,t)+\xi (u(x,t),x)_x=0,
	\end{equation}
	that expresses the mass conservation of $u$ over the moving cell $C(\tau)$. Thus, by the Leibniz rule,
	\begin{equation}\label{LeibRule}
		\frac{d}{d\tau} \int_{x_1(\tau)}^{x_2(\tau)}
		u(x,\tau)dx=u(x_2(\tau),\tau)x_2'(\tau)-u(x_1(\tau),\tau)x_1'(\tau)+\int_{x_1(\tau)}^{x_2(\tau)}u_t(x,\tau)dx.
	\end{equation}
	If the mass of $u$ over $C(\tau)$ remains constant with respect to $\tau$, the following condition holds
	\begin{equation}\label{MassCons}
		\int_{x_1(\tau)}^{x_2(\tau)}u_t(x,\tau)dx=-u(x_2(\tau),\tau)x_2'(\tau)+u(x_1(\tau),\tau)x_1'(\tau). 
	\end{equation}
	Integrating \eqref{GCL1} over $C(\tau)$ we obtain
	$$\int_{x_1(\tau)}^{x_2(\tau)}u_t(x,\tau)dx + \int_{x_1(\tau)}^{x_2(\tau)} \xi(u(x,\tau),x)_x dx=0.$$
	Now, using \eqref{MassCons} we get
	$$ \xi(u(x_2(\tau),\tau),x_2(\tau)) - \xi(u(x_1(\tau),\tau),x_1(\tau))=u(x_2(\tau),\tau)x_2'(\tau)-u(x_1(\tau),\tau)x_1'(\tau).$$
	A suitable flux function $\xi$ hence is
	\begin{equation}\label{ProperFlux}
		\xi(u(x(\tau),\tau),x(\tau))=u(x(\tau),\tau)x'(\tau).
	\end{equation}
	Therefore, the strong formulation of \eqref{GCL1} reads
	\begin{equation}\label{GCL}
		u_t(x,\tau)+(u(x,\tau)x_t)_x=0,
	\end{equation}
	which is referred in the literature as the Geometric Conservation Law (GCL), see e.g. \cite{Trullio.1961}, \cite{Thomas.1979}	, \cite{Huang.2002}.

	As previously announced we can attain the per-cell mass conservation property \eqref{PerCellCons1} by discretizing the corresponding GCL.

	\begin{lemma}\label{PerCellLemma}
		The per-cell mass conservation label \eqref{PerCellCons1} is a consequence of the geometric conservation law \eqref{GCL}.
	\end{lemma}
	\begin{proof}
		For every given cell-value pair $C_i$, $u_i$ and the respective reference pair $\bar C_i$, $v_i$ --as provided in the Definition \ref{RefMesh}-- we set the moving cell 
		$C(\tau)=(x_1(\tau), x_2(\tau))$ for $\tau\in[\tau_1,\tau_2]$ to be a linear interpolation of $C_i$ and $\bar C_i$
		$$C(\tau)=\frac {\tau-\tau_1}{\tau_2-\tau_1} \bar C_i + \frac{\tau_2-\tau}{\tau_2-\tau_1} C_i.$$
		Now, to attain a discrete version of \eqref{GCL} we integrate it over $C(\tau)$ 
		$$\int_{x_1(\tau)}^{x_2(\tau)} u_t(x,\tau) dx  +  \int_{x_1(\tau)}^{x_2(\tau)} \left( u(x(\tau),\tau) x_t(x,\tau) \right)_x dx=0,$$
		and invoke \eqref{LeibRule} to get
		$$\frac{d}{d\tau} \int_{x_1(\tau)}^{x_2(\tau)} u(x,\tau)dx-u(x_2(\tau),\tau)x_2'(\tau)+u(x_1(\tau),\tau)x_1'(\tau)+\int_{x_1(\tau)}^{x_2(\tau)} \left(u(x(\tau),\tau) x_t(x,\tau) \right)_x dx=0.$$
		We discretize explicitly in $[\tau_1, \tau_2]$, set $\Delta \tau=\tau_2-\tau_1$, and recall that $|C_i|=h_i$, and $\bar C_i=\Delta x$ to get
		$$\frac{1}{\Delta \tau}(\Delta x v_i-h_i u_i)-u_i\frac{x_2(\tau_2)-x_2(\tau_1)}{\Delta \tau} + u_i\frac{x_1(\tau_2)-x_1(\tau_1)}{\Delta \tau} 
							+ u_i\frac{x_2(\tau_2)-x_2(\tau_1)}{\Delta \tau}- u_i\frac{x_1(\tau_2)-x_1(\tau_1)}{\Delta \tau}=0$$
		or simply
		$$\Delta x\; v_i = h_i\; u_i.$$
	\end{proof}
	\begin{remark}
		The time variable $\tau$ used in the previous proof refers to ficticious time; it does not correspond to the physically relevant time.
	\end{remark}

	Let us point out that in the theoretical analysis we will use the reference uniform mesh-solution pair to combine the effects of mesh adaptivity and the numerical update. 

	In view of the Definition \ref{RefMesh}, and after applying the per-cell mass conservation \eqref{PerCellCons1} on the schematic representation \eqref{2grid1} of MAS, we 
	get for $i=1,\ldots, N$:
	\begin{equation}\label{PerCellCons}	
		h_i^n\, u_i^n=\Delta x\, v_i^n 	\xrightarrow{\text{mesh adapt.}} h_i^{n+1}\, \hat u_i^n=\Delta x\, \hat v_i^n 
								\xrightarrow{\text{num. scheme}} h_i^{n+1}\, u_i^{n+1}=\Delta x\, v_i^{n+1}.
	\end{equation}

	Now, by invoking \eqref{PerCellCons} after multiplying with $h_i^{n+1}$, we can rewrite the scheme \eqref{NonUniScheme} over the uniform reference mesh  
	\begin{equation}\label{UniScheme}
		v_i^{n+1}=\hat{v}_i^n-\frac{\Delta t}{\Delta x} \left(\hat F_{i+1/2}^n-\hat F_{i-1/2}^n\right),
	\end{equation}
	where $\hat F$ is the numerical flux function from \eqref{NonUniScheme} written in variables $\hat v$. 

%%%%%%%%%%%%%%%%%%%%%%%%%%%%%%%%%%%%%%%%%%%%%%%%%%%%%%%%%%%%%%%%%%%%%%%%%%%%%%%%%%%%%%
%%%%%%%%%%%%%%%%%%%%%%%%%%%%%%%%%%%%%%%%%%%%%%%%%%%%%%%%%%%%%%%%%%%%%%%%%%%%%%%%%%%%%%

\section{Entropy stablility}
	Before stating the main theoretical result we introduce the following notations:
	\begin{subequations}
	\begin{align}
		\Delta v_{i+1/2}^n&:=v_{i+1}^n-v_i^n,\label{1stNot}\\
		B_{i+1/2}^n&:=\frac{f(v_{i+1}^n)-f(v_i^n)}{\Delta v_{i+1/2}^n},\\
		Q_{i+1/2}^n&:=\frac{f(v_{i+1}^n)+f(v_i^n)-2\hat F_{i+1/2}^n}{\Delta v_{i+1/2}^n},\\
		\Delta x_{i+1/2}^{n+1/2}&:=x_{i+1/2}^{n+1}-x_{i+1/2}^{n},,\\
		H_{i+1/2}^n&:=\left( \frac{({\Delta x_{i+1/2}^{n+1/2}})_-}{h_i^n} v_i^n - \frac{({\Delta x_{i+1/2}^{n+1/2}})_+}{h_{i+1}^n} v_{i+1}^n \right),\\
		M_i^n&:=\tilde v_i^n\left(H_{i-1/2}^n - H_{i+1/2}^n\right),\label{lastNot}
	\end{align}
	\end{subequations}
	Now, we proceed with the main theoretical result. 
	\begin{theorem}
		We use the notations \eqref{1stNot}-\eqref{lastNot} and assume that the following condition holds 
		\begin{align}\label{MainCond}
			M_i^n \leq \frac{\Delta t}{4\Delta x}\Bigg\{ &\bigg( D_{i-1/2}^n - K^3\frac{\Delta t}{\Delta x} (B_{i-1/2}+Q^\ast_{i-1/2}+D_{i-1/2})^2 \bigg) (\Delta v_{i-1/2}^n)^2 \\
										+	 &\bigg( D_{i+1/2}^n- K^3\frac{\Delta t}{\Delta x} (B_{i+1/2}-Q^\ast_{i+1/2}-D_{i+1/2})^2 \bigg) (\Delta v_{i+1/2}^n)^2 \Bigg\},
		\end{align}
		where $\Delta x$, $\Delta t$ are respectively the space and time steps that correspond to the numerical scheme \eqref{NonUniScheme}. The mesh adaptation procedure \eqref{2grid1} is used, 
		where the corresponding reference uniform mesh is given in the Definition \ref{RefMesh}. Then the mesh adaptation procedure MAS \eqref{2grid1} with the numerical scheme 
		\eqref{NonUniScheme} for the time evolution step is entropy stable.
	\end{theorem}
	\begin{proof}
		The numerical scheme for the uniform variables reads, cf. \eqref{UniScheme} 
		$$v_i^{n+1}=\hat{v}_i^n-\frac{\Delta t}{\Delta x} \left(\hat F_{i+1/2}^n-\hat F_{i-1/2}^n\right).$$
		We subtract $v_i^n$ to develop the respective incremental form 
		$$v_i^{n+1}-v_i^n=\hat v_i^n -v_i^n -\frac{\Delta t}{2\Delta x} \left(2\hat F_{i+1/2}^n-2\hat F_{i-1/2}^n\right).$$
		Equivalently 
		\begin{subequations}
		\begin{align}
			v_i^{n+1}-v_i^n	&=\hat v_i^n -v_i^n\color{black}-\frac{\Delta t}{2\Delta x} \bigg(f(v_{i+1}^n)+f(v_i^n) -\frac{f(v_{i+1}^n)+f(v_i^n)-2\hat
				F_{i+1/2}^n}{v_{i+1}^n-v_i^n}(v_{i+1}^n-v_i^n)\nonumber\\
						&\hspace{7.5em} -f(v_i^n)-f(v_{i-1}^n)+\frac{f(v_i^n)+f(v_{i-1}^n)-2\hat F_{i-1/2}^n}{v_i^n-v_{i-1}^n}(v_i^n-v_{i-1}^n)\bigg)\label{SemiIncr} \\
						&=	\hat v_i^n -v_i^n-\frac{\Delta t}{2\Delta x} \bigg(\left(B_{i+1/2}^n-Q_{i+1/2}^n\right)\Delta v_{i+1/2}^n + \left(B_{i-1/2}^n+Q_{i-1/2}^n\right)
									\Delta v_{i-1/2}^n\bigg).\label{SemiIncr2}
		\end{align}
		\end{subequations}

		We point out that the term $\hat v_i^n -v_i^n$ is new and accounts for the mesh reconstruction and solution update steps
		of the MAS \eqref{MAS}.\\
		We now express $\hat v_i^n-v_i^n$ in a conservative form with respect to $\left\{v _j^n\right\}$. Accordingly the size of $C_i$ changes as:
		\begin{equation}\label{GenMove}
			h_i^{n+1}=h_i^n +{\Delta x_{i+1/2}^{n+1/2}} - {\Delta x_{i-1/2}^{n+1/2}},
		\end{equation}
		and the the mass of $u$ over $C_i$ as:
		\begin{equation}\label{MassChange1}
			h_i^{n+1} \hat u_i^n =h_i^n u_i^n 	- ({\Delta x_{i+1/2}^{n+1/2}})_- u_i^n + ({\Delta x_{i+1/2}^{n+1/2}})_+ u_{i+1}^n  
										+  ({\Delta x_{i-1/2}^{n+1/2}})_- u_{i-1}^n - ({\Delta x_{i-1/2}^{n+1/2}})_+ u_i^n
		\end{equation}
%%		where $(\cdot)_+=\max\{\cdot,0\}$ and $(\cdot)_-=-\min\{\cdot,0\}$ denote respectively the positive and negative parts. We switch to the uniform mesh using \eqref{PerCellCons1} and 
		\eqref{MassChange1} recasts to:
		$$\hat v_i^n =v_i^n - \frac{({\Delta x_{i+1/2}^{n+1/2}})_-}{h_i^n} v_i^n + \frac{({\Delta x_{i+1/2}^{n+1/2}})_+}{h_{i+1}^n} v_{i+1}^n
						+ \frac{({\Delta x_{i-1/2}^{n+1/2}})_-}{h_{i-1}^n} v_{i-1}^n - \frac{({\Delta x_{i-1/2}^{n+1/2}})_+}{h_i^n}v_i^n.$$
		This relation can also be written as a conservative difference
		\begin{align}
			\hat v_i^n -v_i^n	&=	\left( \frac{({\Delta x_{i-1/2}^{n+1/2}})_-}{h_{i-1}^n} v_{i-1}^n - \frac{({\Delta x_{i-1/2}^{n+1/2}})_+}{h_i^n}v_i^n \right) 
								- \left( \frac{({\Delta x_{i+1/2}^{n+1/2}})_-}{h_i^n} v_i^n - \frac{({\Delta x_{i+1/2}^{n+1/2}})_+}{h_{i+1}^n} v_{i+1}^n \right)\label{MassChange2}\\
							&=	H_{i-1/2}^n- H_{i+1/2}^n\label{MassChange3}
		\end{align}
		Replacing \eqref{MassChange3} in \eqref{SemiIncr2} we obtain
		\begin{align}
			v_i^{n+1}-v_i^n=	&H_{i-1/2}^n- H_{i+1/2}^n\nonumber\\
							&-\frac{\Delta t}{2\Delta x} \bigg(\left(B_{i+1/2}^n-Q_{i+1/2}^n\right)\Delta v_{i+1/2}^n + \left(B_{i-1/2}^n+Q_{i-1/2}^n\right),
									\Delta v_{i-1/2}^n\bigg)\label{SemiIncr3}
		\end{align}
		which can be analogously written as a conservative update over the reference uniform mesh 
		\begin{equation}\label{CombNumScheme}
			v_i^{n+1}=v_i^n-\frac{\Delta t}{\Delta x}\left(\frac{\Delta x}{\Delta t}H_{i+1/2}+\hat F_{i+1/2}^n -\frac{\Delta x}{\Delta t}H_{i-1/2}-\hat F_{i-1/2}^n  \right).
		\end{equation}
		We note that the conservative difference $H_{i-1/2}^n- H_{i+1/2}^n$ accounts for the mesh reconstruction and the solution update step of the MAS.\\
		
		In order to simplify the presentation of the rest of the proof we assume that the entropy and the conservative variables ($\tilde v$ and $v$, respectively) coincide, i.e. we choose 
		$U(u)=\frac{1}{2} u^2$ for the entropy function. Now, to recover the entropy-entropy flux representation of \eqref{SemiIncr3}, we multiply it by the entropy variables $\tilde v_i^n$,
		yielding
		\begin{align}
			U(v_i^{n+1})-U(v_i^n)+\frac{\Delta t}{\Delta x}\big( G_{i+1/2}^n-G_{i-1/2}^n \big)
				&=M_i^n-\frac{\Delta t}{\Delta x}\mathcal E_i^{(x)} + \mathcal E_i^{(FE)}(\Delta v^{n+1/2})\label{FinRel}
		\end{align}
		where $G$ is the numerical entropy flux. We have further following \cite{Tadmor.2003},
		\begin{align*}	
			\mathcal E_i ^{(x)} 		&= \frac{1}{4}\left[D_{i-1/2}\Delta v_{i-1/2}^2 + D_{i+1/2}\Delta v_{i+1/2}^2\right]\\
			\mathcal E_i^{(FE)} 	&\leq \frac{K^3}{4}\left(\frac{\Delta t}{\Delta x}\right)^2\left[ 	(B_{i+1/2}-Q_{i+1/2})^2\Delta v_{i+1/2}^2 + 
																				(B_{i-1/2}+Q_{i-1/2})^2\Delta v_{i-1/2}^2\right],\\
								&\leq \frac{K^3}{4}\left(\frac{\Delta t}{\Delta x}\right)^2\left[ 	(B_{i+1/2}-Q^\ast_{i+1/2}-D_{i+1/2})^2\Delta v_{i+1/2}^2 + 
																				(B_{i-1/2}+Q^\ast_{i-1/2}+D_{i-1/2})^2\Delta v_{i-1/2}^2\right].
		\end{align*}
		For more details see also Appendix \ref{TadmorCorrected}. 
		Now, \eqref{FinRel} reduces to 
		\begin{align*}
			&U(v_i^{n+1})-U(v_i^n)+\frac{\Delta t}{\Delta x}\big( G_{i+1/2}^n-G_{i-1/2}^n \big) \leq M_i^n -\frac{\Delta t}{4\Delta x}\left( D_{i-1/2}\Delta v_{i-1/2}^2 + D_{i+1/2}\Delta v_{i+1/2}^2 \right) \\
								&+ \frac{K^3}{4}\left(\frac{\Delta t}{\Delta x}\right)^2\left[ 	(B_{i+1/2}-Q^\ast_{i+1/2}-D_{i+1/2})^2\Delta v_{i+1/2}^2 + 
																				(B_{i-1/2}+Q^\ast_{i-1/2}+D_{i-1/2})^2\Delta v_{i-1/2}^2\right],
		\end{align*}
		Hence, the sufficiency condition for entropy stability reads 
		\begin{align*}
			M_i^n \leq \frac{\Delta t}{4\Delta x}\Bigg\{ &\bigg( D_{i-1/2}^n - K^3\frac{\Delta t}{\Delta x} (B_{i-1/2}+Q^\ast_{i-1/2}+D_{i-1/2})^2 \bigg) (\Delta v_{i-1/2}^n)^2 \\
										+	 &\bigg( D_{i+1/2}^n- K^3\frac{\Delta t}{\Delta x} (B_{i+1/2}-Q^\ast_{i+1/2}-D_{i+1/2})^2 \bigg) (\Delta v_{i+1/2}^n)^2 \Bigg\}.
		\end{align*}
	\end{proof}
	
	\begin{remark}
		We point out that numerical scheme \eqref{CombNumScheme}
		$$v_i^{n+1}=v_i^n-\frac{\Delta t}{\Delta x}\left(\frac{\Delta x}{\Delta t}H_{i+1/2}+\hat F_{i+1/2}^n -\frac{\Delta x}{\Delta t}H_{i-1/2}-\hat F_{i-1/2}^n  \right),$$
		is written using uniform variables and incorporates both the time evolution step and the adaptation of the mesh.
	\end{remark}

	\begin{figure}[t]
		\centering
		\subfigure[The cell $i$ moves to the left] {\includegraphics[width=0.4\linewidth]{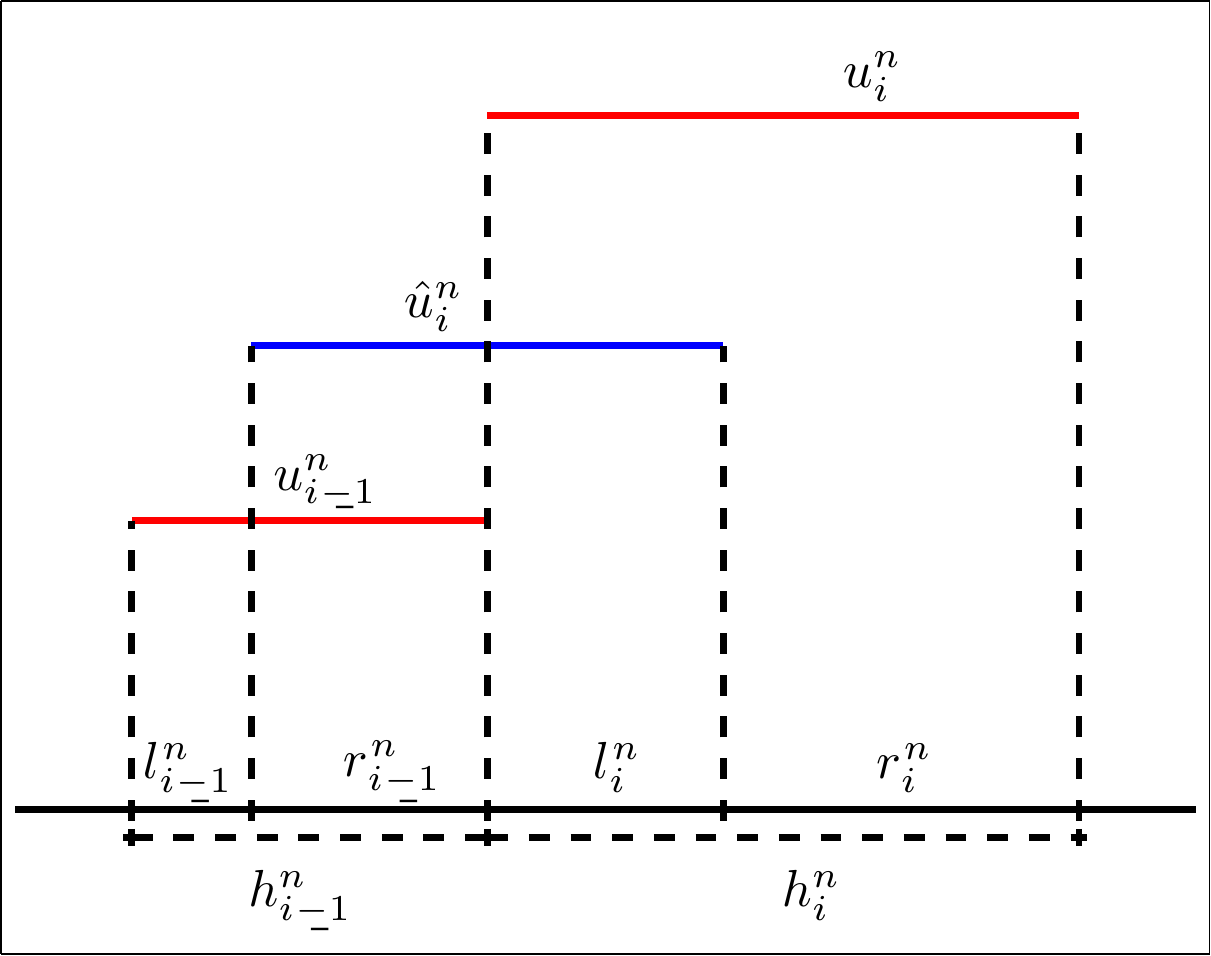}} 
		\subfigure[The cell $i$ moves to the right]{\includegraphics[width=0.4\linewidth]{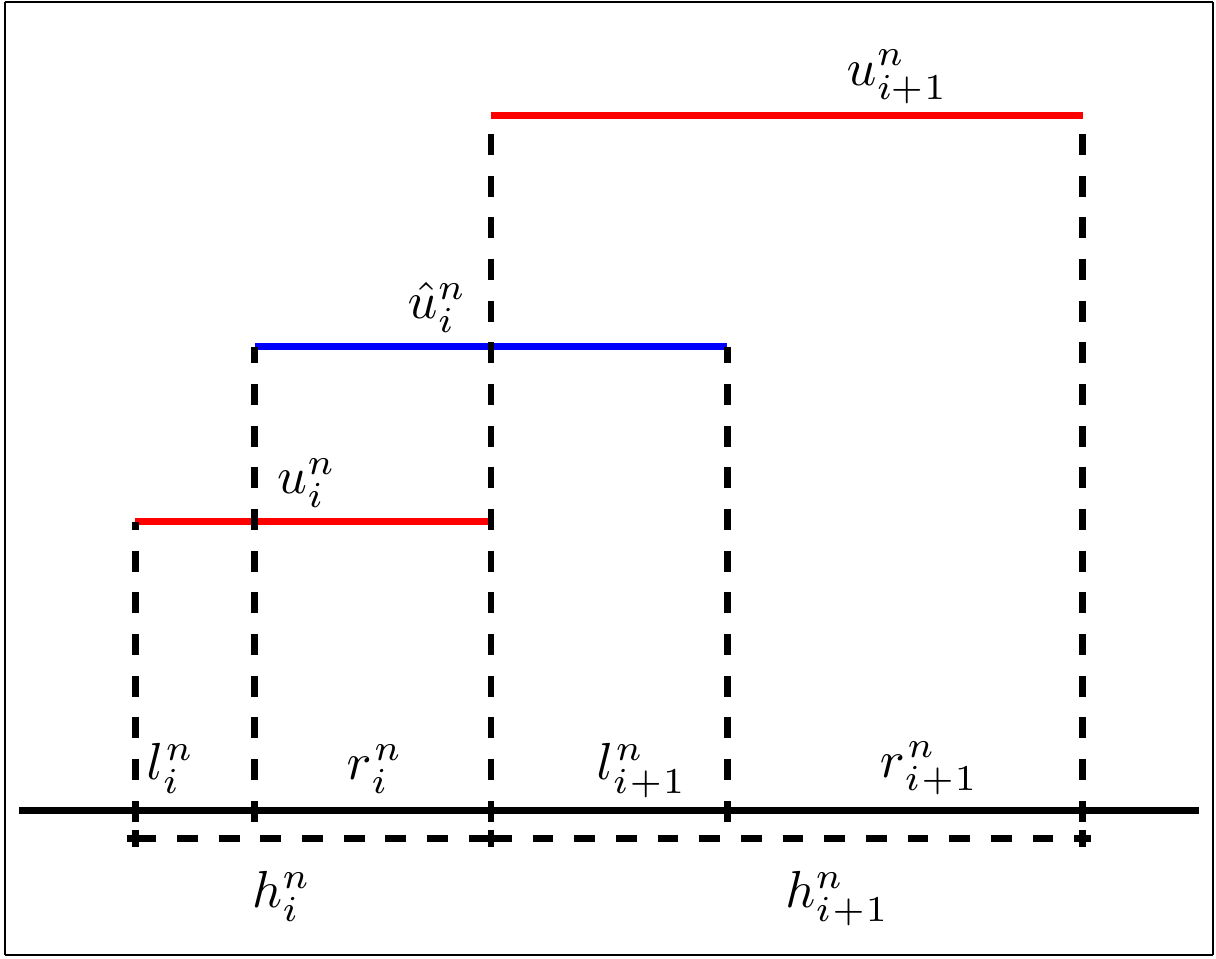}}\\
		\subfigure[The cell $i$ moves to both directions]{\includegraphics[width=0.4\linewidth]{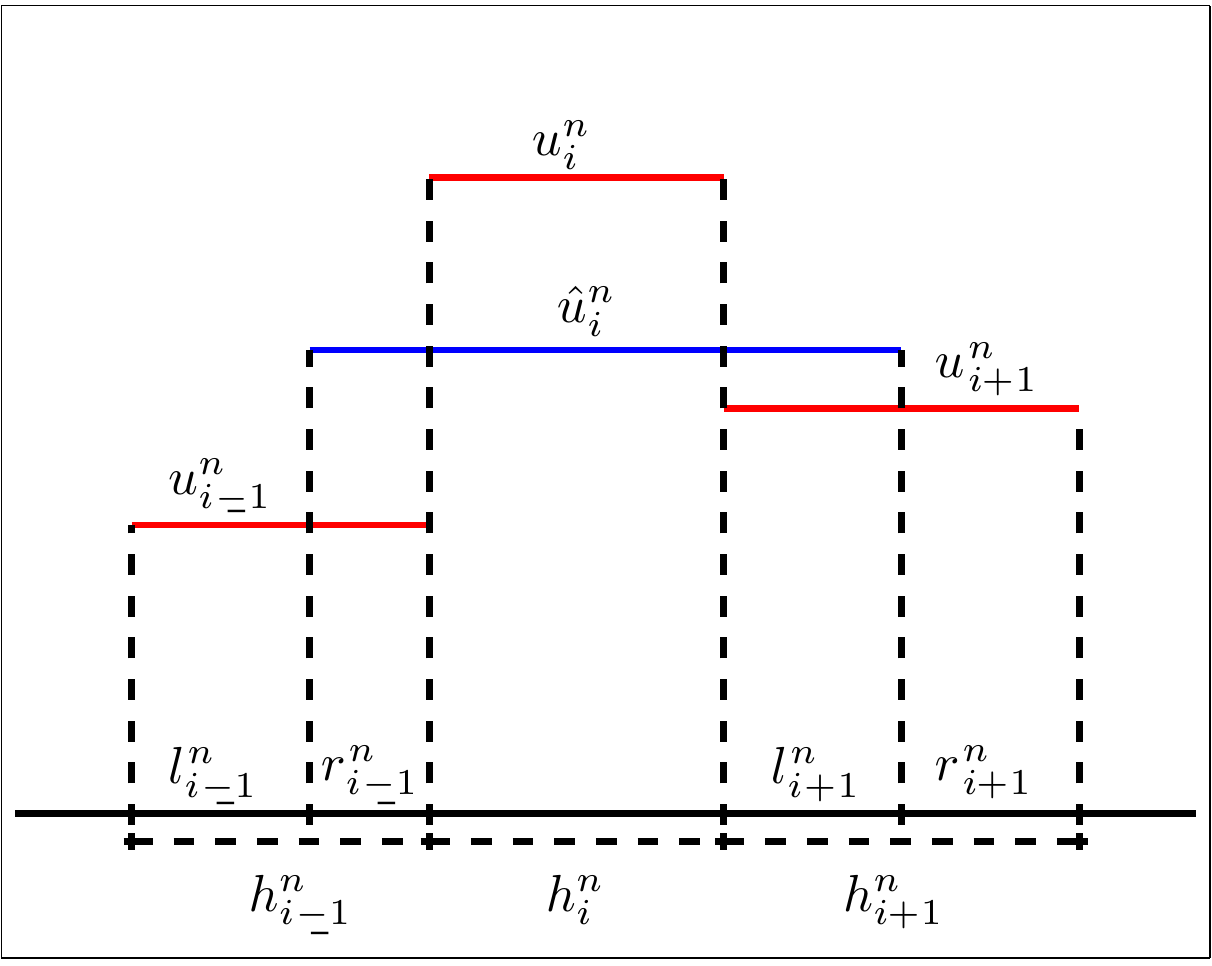}}
		\caption{Three cases of cell movement.}\label{FigMeshMove}
	\end{figure}

	\begin{example}
		To gain further insight into \eqref{GenMove}, \eqref{MassChange3} we refer to Figure \ref{FigMeshMove} and provide three special examples. 
		\begin{itemize}
			\item Cell moves to the left.\\
				This means that ${\Delta x_{i+1/2}^{n+1/2}}=-r_i^n$, ${\Delta x_{i-1/2}^{n+1/2}}=-r_{i-1}^n$, so $h_i^{n+1}=h_i^n -r_i^n+r_{i-1}^n=r_{i-1}^n +l_i^n$.
				The mass of $u$ satisfies $h_i^{n+1} \hat u_i^n=r_{i-1}^n u_{i-1}^n+l_i^nu_i^n$; passing in $v$ variables 
				$\hat v_i^n =\frac{r_{i-1}^n}{h_{i-1}^n}v_{i-1}^n+\frac{l_i^n}{h_i^n} v_i^n$ or 
				\begin{equation}\label{Exam_1}
					\hat v_i^n-v_i^n =\frac{r_{i-1}^n}{h_{i-1}^n}v_{i-1}^n - \frac{r_i^n}{h_i^n} v_i^n.
				\end{equation}
			\item Cell moves to the right.\\
				This means that ${\Delta x_{i+1/2}^{n+1/2}}=l_{i+1}^n$, ${\Delta x_{i-1/2}^{n+1/2}}=l_i^n$, so $h_i^{n+1}=h_i^n +l_{i+1}^n-l_i^n=r_i^n +l_{i+1}^n$. In this case,
				the mass of $u$ satisfies $h_i^{n+1} \hat u_i^n=r_i^n u_i^n+l_{i+1}^nu_{i+1}^n$, or in $v$ variables 
				$\hat v_i^n=\frac{r_i^n}{h_i^n}v_i^n+\frac{l_{i+1}^n}{h_{i+1}^n} v_{i+1}^n$, or
				\begin{equation}\label{Exam_2}
					\hat v_i^n-v_i^n =-\frac{l_i^n}{h_i^n}v_i^n + \frac{l_{i+1}^n}{h_{i+1}^n} v_{i+1}^n.
				\end{equation}
			\item Cell moves to both directions.\\
				Similarly, ${\Delta x_{i+1/2}^{n+1/2}}=l_{i+1}^n$, ${\Delta x_{i-1/2}^{n+1/2}}=-r_{i-1}^n$, so $h_i^{n+1}=h_i^n +l_{i+1}^n+r_{i-1}^n$. Moreover 
				$h_i^{n+1} \hat u_i^n=r_{i-1}^n u_{i-1}^n + h_i^n u_i^n + l_{i+1}^n u_{i+1}^n$, or $\hat v_i^n =\frac{r_{i-1}^n}{h_{i-1}^n}v_{i-1}^n + v_i^n +\frac{l_{i+1}^n}{h_{i+1}^n} v_{i+1}^n$, 
				or
				\begin{equation}\label{Exam_3}
					\hat v_i^n -v_i^n =\frac{r_{i-1}^n}{h_{i-1}^n}v_{i-1}^n + \frac{l_{i+1}^n}{h_{i+1}^n} v_{i+1}^n.
				\end{equation}
		\end{itemize}
		Let us point that now the conditions \eqref{Exam_1}-\eqref{Exam_3} are simple enough in order to be tested for a moving mesh algorithm. We will report numerical experiments in a future work. 
	\end{example}

%%%%%%%%%%%%%%%%%%%

%%%%%%%%%%%%%%%%%%%%%%%%%%%%%%%%%%%%%%%%%%%%%%%%%%%%%%%%%%%%%%%%%%%%%%%%%%%%%%%%%%%%%%
%%%%%%%%%%%%%%%%%%%%%%%%%%%%%%%%%%%%%%%%%%%%%%%%%%%%%%%%%%%%%%%%%%%%%%%%%%%%%%%%%%%%%%
\section{Conclusions}
	We provide in this work a new framework of studying the combined effect of mesh adaptation and time evolution of a numerical solution. This new method, has the benefit of
	being described over a uniform ``underlying'' grid that resolves the physical domain with the same number of discretization nodes as the numerical solution itself. To exhibit 
	properties of this technique we study the dissipation of entropy due to the adaptation of the mesh. We have derived a sufficient condition for the mesh movement in order 
	to guarantee that the overall procedure dissipates the entropy. Consequently the resulting numerical scheme MAS \eqref{MAS}.

%%%%%%%%%%%%%%%%%%%%%%%%%%%%%%%%%%%%%%%%%%%%%%%%%%%%%%%%%%%%%%%%%%%%%%%%%%%%%%%%%%%%%%
%%%%%%%%%%%%%%%%%%%%%%%%%%%%%%%%%%%%%%%%%%%%%%%%%%%%%%%%%%%%%%%%%%%%%%%%%%%%%%%%%%%%%%

\vfill
\paragraph{Aknowledgements}~\\
	The authors wish to thank E. Tadmor and Ch. Makridakis for the very useful discussions and suggestions. This work has been partially supported by the research center of 
	Computational sciences in Mainz as well as by the Humboldt foundation. The authors gratefully acknowledge this support. 

\appendix
\section{Entropy stable schemes}\label{TadmorCorrected}
	In this section, for the sake of completeness, we present parts -after some modifications- of the analysis conducted in \cite{Tadmor.1987, Tadmor.2003}. We moreover note 
	that these works are to be viewed in the context of the seminal works on the subject by \cite{Lax.1971, Mock.1978}.
	
	A finite volume approximation of the one-dimensional conservation law is written in the form
	$$u_i ^{n+1}=u_i^n -\frac{\Delta t}{\Delta x}\left( F_{i +1/2}(u^n) - F_{i-1/2}(u^n)\right ).$$
	The corresponding viscosity form reads
	\begin{align}\label{ViscousForm}
		u_i ^{n+1}-u_i^n =\frac{\Delta t}{2\Delta x}\bigg\{ 	&	-\Big( f(u_{i+1}^n)- f(u_u^n) \Big) +Q_{i+1/2}^n \Delta u_{i+1/2}^n \nonumber \\
												&	-\Big( f(u_{i}^n)- f(u_{i-1}^n) \Big) -Q_{i -1/2}^n \Delta u_{i-1/2}^n \bigg\}
	\end{align}
	with the numerical viscosity given by $Q_{i+1/2}^n = \frac{f(u_i^n)+f(u_{i+1}^n)-2F_{i+1/2}^n}{\Delta u_{i+1/2}^n}$. After denoting $B_{i+1/2} =\frac{f(u_{i+1}^n)-f(u_i^n)}{\Delta u_{i+1/2}^n}$
	the equation \eqref{ViscousForm} can be rewritten as 
	$$u_i ^{n+1} - u_i^n 	=-\frac{\Delta t}{2\Delta x}\bigg( \Big( B_{i+1/2} - Q_{i +1/2} \Big) \Delta u_{i+1/2}^n + \Big( B_{i-1/2} + Q_{i -1/2} \Big) \Delta u_{i-1/2}^n \bigg).$$
	Now, setting $D_{i+1/2}=Q_{i+1/2}-Q_{i+1/2}^\ast$, with $Q^\ast$ the numerical viscosity of an entropy conservative scheme we obtain
	$$u_i ^{n+1}-u_i^n=-\frac{\Delta t}{2\Delta x}\bigg( 	\Big( B_{i+1/2} - Q_{i +1/2}^\ast - D_{i +1/2} \Big)  \Delta v_{i+1/2}^n 
				+	\Big( B_{i-1/2} + Q_{i -1/2}^\ast + D_{i -1/2} \Big)  \Delta u_{i-1/2}^n \bigg).$$
	Next we multiply the last equation with the entropy variables $\tilde v$. After noting that
	\begin{align*}
		\left< \tilde v_i^n,u_i^{n+1}-u_i^n \right> = U(u_i^{n+1}) -U(u_i^n)-\mathcal E_i^{(FE)}(\tilde v_i^{n+1/2}),\\
		\left< \tilde v_i^n,F_{i+1/2}^n-F_{i-1/2}^n\right>=G_{i+1/2}^n-G_{i-1/2}^n + \mathcal E_i^x(\tilde v_i^n),
	\end{align*}
	where $\mathcal E_i^{(FE)}(\tilde v_i^{n+1/2})$, $\mathcal E_i^x(\tilde v_i^n)$ are given by 
	\begin{align*}
		\mathcal E_i^{(FE)}(\tilde v_i^{n+1/2}) &=\int_{\xi=-1/2}^{\xi=1/2}\left( \frac 1 2 +\xi \right )\left<\Delta \tilde v_i^{n+1/2} , H(\tilde v_i^{n+1/2}) \Delta \tilde v_i^{n+1/2}\right>,\\
		\mathcal E_i^x(\tilde v_i^n) &=\frac{1}{4} \left< \Delta \tilde v_{i-1/2}^n, D_{i-1/2}^n\Delta \tilde v_{i-1/2}^n \right> + \frac{1}{4} \left< \Delta \tilde v_{i+1/2}^n, D_{i+1/2}\Delta \tilde v_{i+1/2}^n \right>,
	\end{align*}
	we arrive to the entropy-entropy flux pair representation
	$$U(v_i^{n+1}) -U(v_i^n) + \frac{\Delta t}{\Delta x} \bigg\{F_{i+1/2}^n-F_{i-1/2}^n\bigg\} = \mathcal E_i^{(FE)}(\tilde v^{n+1/2})-\frac{\Delta t}{\Delta x} \mathcal E_i^x(\tilde v^n).$$
	Now, entropy stability implies that:
	\begin{equation}\label{EntrCond}
		 \mathcal E_i^{(FE)}(\tilde v^{n+1/2})-\frac{\Delta t}{\Delta x} \mathcal E_i^x(\tilde v^n)\leq 0.
	\end{equation}
	In order to satisfy \eqref{EntrCond} the following conditions are sought as sufficient
	\begin{align}
		K^3 \frac{\Delta t}{\Delta x} (B_{i+1/2}-Q_{i+1/2}^\ast-D_{i+1/2})^2\leq D_{i+1/2}\nonumber\\
		K^3 \frac{\Delta t}{\Delta x} (B_{i-1/2}+Q_{i-1/2}^\ast+D_{i-1/2})^2\leq D_{i-1/2}\nonumber
	\end{align}
	for every $i$. Due to symmetry they give
	\begin{equation}
		K^3 \frac{\Delta t}{\Delta x} (B_{i+1/2} \pm Q_{i+1/2}^\ast \pm D_{i+1/2})^2\leq D_{i+1/2}
	\end{equation}
	The last relation yields the following results for $c=\frac{\Delta x}{K^3 \Delta t}$
	\begin{align*}
		D_{i+1/2}&\geq 0\\
		(B\pm Q^\ast \pm D_{i+1/2})^2&\leq c D_{i+1/2}.
	\end{align*}
	This leads to $D_{i+1/2}^2+(2(Q^\ast\pm B)-c)D_{i+1/2}+(Q^\ast\pm B)^2\leq 0$. By setting $k_1=Q^\ast + B$ and $k_2=Q^\ast - B$ we get two different quadratic inequalities 
	that need to be satisfied for $D_{i+1/2}$:
	$$D_{i+1/2}^2+(2k_1-c)D_{i+1/2}+k_1^2\leq 0 \text{ and } D_{i+1/2}^2+(2k_2-c)D_{i+1/2}+k_2^2\leq 0.$$
	The necessary restrictions for the existence of a solution are $c\geq 4k_1$ and $c\geq 4k_2$. This mean that  $Q^\ast\pm B\leq c/4$,  i.e $c\geq 4Q^\ast$ or 
	$$\frac{\Delta t }{\Delta x}\leq \frac{1}{ 4K^3 Q^\ast}.$$

%%%%%%%%%%%%%%%%%%%%%%%%OLD Stuff
%%%%%%%%%%%%%%%%%%%%%%%%%%%%%%%%%%%%%%%%%%%%%%%%%%%%%%%%%%%%%%%%%%%%%%%%%%%%%%%%%%%%%%
%%%%%%%%%%%%%%%%%%%%%%%%%%%%%%%%%%%%%%%%%%%%%%%%%%%%%%%%%%%%%%%%%%%%%%%%%%%%%%%%%%%%%%
%%%%%%%%%%%%%%%%%%%%%%%%%%%%%%%%%%%%%%%%%%%%%%%%%%%%%%%%%%%%%%%%%%%%%%%%%%%%%%%%%%%%%%
%%%%%%%%%%%%%%%%%%%%%%%%%%%%%%%%%%%%%%%%%%%%%%%%%%%%%%%%%%%%%%%%%%%%%%%%%%%%%%%%%%%%%%

%%%%%%%%%%%%%%%%%%%%%%%%%%%%%%%%%%%%%%%%%%%%%%%%%%%%%%%%%%%%%%%%%%%%%%%%%%%%%%%%%%%%%%
%%%%%%%%%%%%%%%%%%%%%%%%%%%%%%%%%%%%%%%%%%%%%%%%%%%%%%%%%%%%%%%%%%%%%%%%%%%%%%%%%%%%%%

\nocite{*}
\bibliographystyle{agsm}%{plain}%{amsalpha}
\bibliography{LukSfak}
\end{document}